\documentclass[a4paper]{amsart}
\usepackage{amssymb,amsmath,amsfonts,amsthm,mathtools}
\usepackage{enumerate}
\usepackage{microtype}
\usepackage[T1]{fontenc}
\usepackage[english]{babel}
\usepackage[height=24.9cm,a4paper,hmargin={3cm,3cm},top=2.8cm]{geometry}
\usepackage{hyperref}
\mathtoolsset{showonlyrefs}
\numberwithin{equation}{section}
\newtheorem{theorem}{Theorem}[section]
\newtheorem{lemma}[theorem]{Lemma}
\newtheorem{corollary}[theorem]{Corollary}
\theoremstyle{definition}
\newtheorem{definition}[theorem]{Definition}
\theoremstyle{remark}
\newtheorem{remark}[theorem]{Remark}
\newtheorem{example}[theorem]{Example}
\newcommand{\RR}{\mathbb{R}}
\newcommand{\CC}{\mathbb{C}}
\newcommand{\NN}{\mathbb{N}}

\newcommand{\sphere}{\mathbb{S}}
\newcommand{\cN}{\mathcal{N}}

\DeclarePairedDelimiter{\abs}{|}{|}
\newcommand{\indic}{\mathbf{1}}
\DeclareMathOperator{\Ran}{Ran}
\DeclareMathOperator*{\supp}{supp}
\newcommand*\Diff[1]{\mathop{}\!\mathrm{d}#1}
\newcommand{\norm}[1]{\ensuremath{\left\|#1\right\|}}
\newcommand{\bes}{\begin{equation*}}
\newcommand{\ees}{\end{equation*}}
\newcommand{\be}{\begin{equation}}
\newcommand{\ee}{\end{equation}}
\newcommand{\eqs}[1]{\begin{align*}#1\end{align*}}

\let\Re\relax
\DeclareMathOperator{\Re}{Re}
\let\Im\relax
\DeclareMathOperator{\Im}{Im}

\newcommand{\trace}{\mathrm{image}}
\renewcommand{\restriction}{|}
\newcommand{\euler}{\mathrm{e}}
\newcommand{\obs}{\mathrm{obs}}
\newcommand{\arc}[1]{\ensuremath{\sigma(#1)}}
\newcommand{\diam}{\mathrm{diam}}
\title[Logvinenko-Sereda-Kovrijkine type inequality on the sphere]
{Spherical Logvinenko-Sereda-Kovrijkine type inequality and null-controllability of the heat equation on the sphere}
\hypersetup{
	pdftitle={Logvinenko-Sereda-Kovrijkine type inequality on the sphere and null-controllability of the spherical heat equation},
	pdfauthor={Alexander Dicke, Ivan Veseli\'c},
	hidelinks
}
\author[A.~Dicke]{Alexander Dicke}
\address[A.D.]{
	Dortmund, Germany
}
\email{adicke.math@gmail.com}
\author[I.~Veseli\'c]{Ivan Veseli\'c}
\address[I.V.]{
	Technische Univer\-si\-t\"at Dortmund, Germany
}
\urladdr{\url{https://www.mathematik.tu-dortmund.de/lsix/}}
\email{iveselic@mathematik.tu-dortmund.de}
\thanks{\jobname.tex, \today, \copyright with the authors.}
\begin{document}
%
%
\begin{abstract}
	It is shown that the restriction of a polynomial to a sphere satisfies a Logvinenko-Sereda-Kovrijkine type inequality
	(a specific type of uncertainty relation).
	This implies a spectral inequality for the Laplace-Beltrami operator, which, in turn, yields
	observability and null-controllability with explicit estimates on the control costs for the spherical heat equation that
	are sharp in the large and in the small time regime.
\end{abstract}
\maketitle
%
%
\section{Introduction and results}

This paper has two purposes.
The first is to establish an uncertainty relation for spherical polynomials.
Due to its analogy to Kovrijkine's version of the so-called Logvinenko-Sereda
inequality we call it a \emph{Logvinenko-Sereda-Kovrijkine type inequality}.
It can be translated to a \emph{spectral inequality}, as it is often used in control theory,
for the Laplace-Beltrami operator on the sphere $\sphere^{d-1}\subset \RR^d$.
This leads to the second purpose of this work, namely observability inequalities and control cost estimates for the spherical heat equation.

Although
this paper is not devoted to
applications, let us mention that observability estimates on $\sphere^{d-1}$ would be of interest
if one wants to observe heat propagation on the surface of a body of approximately spherical shape where sensors can only be placed on a (small) portion of the object.
In the context of climate models one can think of the observation of the global heat flow on the earth's surface,
albeit it is known in climatology that convective and radiation effects dominate the diffusion process.
In any case, in the study of controllability of the heat equation in non-Euclidean domains
the sphere is a prime example.

As mentioned, we first turn our attention to the uncertainty principle,
which has many physical and mathematical manifestations.
All of them imply that it is impossible for a compactly supported function to have a compactly supported Fourier transform.
We are here concerned with a quantitative version of this statement.
Before considering this question on the sphere, we first describe our motivation, namely, previous results in euclidean geometry.
They concern inequalities of the form
\be\label{eq:Logvinenko-Sereda-RRd}
	\norm{f}_{L^2(\RR^d)}\leq C\norm{f}_{L^2(S)}\quad\text{for all}\quad f\in L^2(\RR^d)\quad\text{with}\quad\supp\hat{f}\subset\Sigma,
\ee
for certain sets $S,\Sigma\subset\RR^d$, with a constant $C$ depending only on $S$, $\Sigma$, and the dimension $d$.
If $\Sigma = B(0,r)$ is the ball with radius $r>0$, the sets $S$ for which \eqref{eq:Logvinenko-Sereda-RRd} holds were thoroughly
studied by Panejah \cite{Panejah-61,Panejah-62}, Kacnel'son \cite{Kacnelson-73}, and Logvinenko-Sereda \cite{LogvinenkoS-74}.
Later and using a different approach, Kovrijkine improved the bounds derived by Kacnel'son and Logvinenko-Sereda
and obtained the optimal constant, for the one-dimensional case in  \cite{Kovrijkine-01}
and for the multi-dimensional one in his thesis \cite{Kovrijkine-thesis}.

\begin{theorem}[Panejah, Kacnel'son, Logvinenko-Sereda, Kovrijkine]\label{thm:Logvinenko-Sereda-RRd}
	Let $S\subset\RR^d$ be a measurable set, and let $\Sigma=B_r(0)$ for $r>0$.
	Then \eqref{eq:Logvinenko-Sereda-RRd} holds with some $C>0$, if and only if $|S\cap (x+(0,a)^d)|\geq \gamma a^d$
	for some $\gamma,a>0$.
	If this is the case, one can choose $C=(c_0^d/\gamma)^{d(c_0ar+1)}$ with a universal constant $c_0>0$.
\end{theorem}

A set $S$ satisfying the geometric condition of the last theorem is
referred to as a \emph{thick set}.

In the present paper we adapt Kovrijkine's approach to prove an uncertainty
relation on the sphere, dubbed here a Logvinenko-Sereda-Kovrijkine type inequality for
\emph{spherical polynomials}, more precisely for functions $f\colon \sphere^{d-1}_R:= \{x\in\RR^d\colon |x|=R\}\to \CC$
such that there exists a polynomial $P\colon \RR^d\to \CC$ with $f=P\restriction_{\sphere^{d-1}_R}$.
We say that $f$ has \emph{degree at most} $\cN\in\NN$, if there is a corresponding polynomial $P$ of degree $\cN$ (or less) on $\RR^d$
such that the last equality holds.

In this setting we introduce the notion of thickness with respect to \emph{spherical caps} $K(x,a) $ of radius $a>0$.

Since the spherical distance between two points $u,v\in\sphere^{d-1}_R$ is given by $d_R(u,v)=R\arccos(u\cdot v/R^2)$, these are sets of the form
\[
	K(x,a) = \{y\in \sphere^{d-1}_R\colon d_R(x,y)\leq \pi a\} \quad\text{for some}\quad x\in\sphere^{d-1}_R.
\]

\begin{definition}[Thick sets on $\sphere^{d-1}_R$]\label{def:thick}
	Let $0<\gamma\leq 1$.
	A measurable set $S\subset\sphere^{d-1}_R$ is called \emph{$\gamma$-thick}, if there exists some radius $a>0$ such that
	$|S\cap K|\geq \gamma|K|$ for all spherical caps $K$ of radius $a$.
\end{definition}

Throughout this paper we denote by $c_j > 0$, $j\in\NN$, constants depending only on the dimension $d$.
With this convention, our main result reads as follows.

\begin{theorem}[Logvinenko-Sereda-Kovrijkine type inequality on $\sphere^{d-1}_R$]\label{thm:logvinenko-sereda-sphere}
	Let $1\leq q\leq\infty$, let $f$ be a spherical polynomial of degree at most $\cN\in\NN$,
	and let $S\subset\sphere^{d-1}_R$ be $\gamma$-thick.
	Then
	\be\label{eq:ucp}
		\norm{f}_{L^q(\sphere^{d-1}_R)}\leq\Bigl(\frac{c_1}{\gamma}\Bigr)^{2\cN+1/q} \norm{f}_{L^q(S)}
		.
	\ee
\end{theorem}

\begin{remark}
	A situation that is closely related to the previous theorem in the case $q=2$ was studied in \cite[Corollary~1.1]{MarzoO-08} where,
	however, the constant is not explicitly given.
	In particular, it lacks the explicit dependence on the degree $\cN$ of the polynomial, which is crucial for applications in control theory
	discussed in Theorem \ref{thm:observability-control-cost-estimate} and Remark \ref{rmk:observability-null-control}.
\end{remark}

While the interest in Logvinenko-Sereda-Kovrijkine inequalities stems from abstract considerations in functional and Fourier analysis,
such estimates have applications in areas of applied analysis. For this purpose it is convenient to translate them into a
Hilbert space setting and assume in the following discussion that $q=2$.
Furthermore, let $M$ be a non-empty set, $H$ a non-negative selfadjoint operator on
$L^2(M)$, and let $P_H((-\infty,E]):= \indic_{(-\infty,E]}(H)$ be
the associated spectral projection up to energy $E \in \RR$.

We say that the \emph{spectral projectors $P_H((-\infty,E])$ satisfy a
unique continuation estimate} or that \emph{$H$ satisfies  a spectral inequality}
from some measurable set $S\subset M$, if
for all energies $E\geq0$ there is a constant $C(E)>0$, depending only on
$M$, $S$, $H$, and $E$, such that
\be\label{eq:Logvinenko-Sereda-operator}
	\norm{f}_{L^2(M)}\leq C(E) \norm{f}_{L^2(S)}
	\quad\text{for all}\quad f\in \Ran P_H((-\infty,E])
	.
\ee
The first nomenclature is common in the literature on periodic and random Schr\"odinger operators,
cf., e.g., \cite{GerminetK-13,NakicTTVS-20} and the references therein.
In the following we use the second convention since it
is the established term in the control theory literature, see the references below.
(In other areas of mathematics one encounters further names for such uncertainty relations.)

Clearly, Theorem~\ref{thm:Logvinenko-Sereda-RRd} implies that $H=-\Delta$ on $M=\RR^d$ satisfies a spectral inequality for all thick sets $S$
(since $f\in\Ran P_{-\Delta}((-\infty,E])$ if and only if $\supp\hat{f}\subset B(0,\sqrt{E})$) and the resulting constant is best possible.
The analog is true for
$-\Delta=-\Delta_{\sphere^{d-1}_R}$
in Theorem~\ref{thm:logvinenko-sereda-sphere}:

\begin{corollary}\label{cor:spectral-inequality-Laplace-Beltrami}
	For all $\gamma$-thick sets $S\subset\sphere^{d-1}_R$ with $\gamma \in (0,1]$,
	all energies $E \geq 0$, and all $f\in \Ran P_{-\Delta}((-\infty,E])$ we have
	\be\label{eq:Logvinenko-Sereda-Laplace-Beltrami}
		\norm{f}_{L^2(\sphere^{d-1}_R)}\leq \Bigl(\frac{c_1}{\gamma}\Bigr)^{RE^{1/2}+1/2}\norm{f}_{L^2(S)}
		.
	\ee
\end{corollary}
The proof of Corollary~\ref{cor:spectral-inequality-Laplace-Beltrami} is given in Subsection~\ref{ssec:spectral-inequality} below.

Spectral inequalities play a crucial role in the celebrated \emph{Lebeau-Robbiano method} developed in \cite{LebeauR-95}
to prove observability and null-controllability for the heat equation.
(Some aspects of the proof of the spectral inequality are more accessibly described in \cite{LebeauZ-98,JerisonL-99}.)
There is an extended literature on the Lebeau-Robbiano approach in various settings.
Amongst others, the papers \cite{Miller-10, TenenbaumT-11, NakicTTV-20} are closely related to the situation at hand here,
due to the fact that they consider selfadjoint operators and strive to give efficient and explicit estimates on the constant in the
observability inequality.
More specifically, Kovrijkine's (spectral) inequality in Theorem \ref{thm:Logvinenko-Sereda-RRd}
was used as an input to the Lebeau-Robbiano method
to establish a sharp geometric criterium for null-controllability of the heat equation
on $\RR^d$, see \cite{EgidiV-18, WangWZZ-19}.

A crucial aspect of the inequality \eqref{eq:Logvinenko-Sereda-Laplace-Beltrami}
is that the right hand side depends on the energy only through the factor $E^{1/2}$ in the exponent.
In fact, with $d_0 = (c_1/\gamma)^{1/2}$ and $d_1 = R\log(c_1/\gamma)$ we have
\be\label{eq:spectral-inequality-const-for-obs}
	\Bigl(\frac{c_1}{\gamma}\Bigr)^{\frac{1}{2} + RE^{1/2}}
	\leq
	d_0\euler^{d_1E^{1/2}}.
\ee
Since the dependence on the energy is sublinear,
Corollary~\ref{cor:spectral-inequality-Laplace-Beltrami} (with the specific
bound \eqref{eq:spectral-inequality-const-for-obs}) implies observability of the spherical heat equation
by the Lebeau-Robbiano method.
According to the current state of knowledge, it seems that the formulation of the
last mentioned method in \cite[Theorem~2.8]{NakicTTV-20}
gives the most precise observability estimate:

\begin{theorem}[Observability]\label{thm:observability-control-cost-estimate}
	Let $S\subset\sphere_R^{d-1}$ be a $\gamma$-thick set.
	Then
	\bes
		\norm{\euler^{T\Delta}g}_{L^2(\sphere^{d-1})}^2
		\leq C_\obs^2
		\int_0^T\norm{\euler^{t\Delta}g}_{L^2(S)}^2\Diff{t}
	\ees
	for all $g \in L^2(\sphere^{d-1}_R)$ and all $T > 0$ where
	\[
		C_\obs^2:=\frac{c_2}{T}\exp\Bigl( c_2 \Bigl(\frac{R^2|\log\gamma|^2}{T}+|\log\gamma|\Bigr)\Bigr)
		.
	\]
\end{theorem}

\begin{remark}\label{rmk:observability-null-control}
	For background in control theory suitable for our context we refer the reader for instance to
	\cite{Zuazua-07,LeRousseauL-12,EgidiNSTTV-20,LaurentL-21}.
	In particular, it is well known that by duality between observability and null-con\-trol\-la\-bil\-ity,
	Theorem~\ref{thm:observability-control-cost-estimate}
	implies null-con\-trol\-la\-bil\-ity of the spherical heat equation
	\be\label{eq:spherical-heat}
		\partial_t u - \Delta u = \indic_S f
		,\quad
		u(0) = u_0 \in L^2(\sphere^{d-1}_R)
		,
	\ee
	in time $T > 0$
	with an explicit estimate for the control costs if $S$ is $\gamma$-thick.
	More precisely, for any given initial datum $u_0\in L^2(\sphere^{d-1}_R)$ there is $f \in L^2((0,T);L^2(\sphere^{d-1}_R))$ such
	that $\norm{f}^2 \leq C_\obs^2\|u_0\|^2$ and the mild solution $u$ of \eqref{eq:spherical-heat} satisfies $u(T) = 0$ with $C_\obs$
    as in Theorem \ref{thm:observability-control-cost-estimate}.
\end{remark}

Null-controllability for parabolic equations with distributed controls on measurable sets of positive
measure has been studied before on subsets of the Euclidean space $\RR^d$, e.g., in \cite{ApraizE-13}
(using different methods than ours), however, there the dependence of the control cost on the underlying geometric parameters is not exhibited.
In this vein, see also \cite{EscauriazaMZ-15,EscauriazaMZ-17,BurqM-23,LebeauM}.

In the case where $S$ is a spherical cap, observability of the heat semigroup on
spheres has already been studied in \cite{LaurentL-21}. (In fact, there heat conduction on general manifolds is considered.)
We improve this bound, on one hand, on the qualitative level:
Since the approach of \cite{LaurentL-21} is based on Carleman estimates, it does not seem possible to generalize it to measurable sets $S$.
Our result is, in contrast, applicable to any measurable set $S \subset \sphere^{d-1}_R$ with strictly positive Lebesgue measure,
since such sets are $\gamma$-thick with $\gamma=\abs{S}/\abs{\sphere^{d-1}_R}$.
On the other hand, we impove the observability constant of \cite{LaurentL-21} quantitatively as well.
Namely, for all sensor sets $S$ of fixed volume $V=|S|>0$ we obtain a uniform observability constant.
This is in particular true, if this volume is distributed to a large number of disjoint spherical caps of the same radius.
Letting this number tend to infinity, the radius of each ball tends
to zero and in this regime the observability constant of \cite{LaurentL-21} diverges.

In the large time regime our constant $C_\obs$ decays with $1/\sqrt{T}$, which is optimal according to \cite[Theorem~2.13]{NakicTTV-20},
and which has not been established before to our best knowledge.
Concerning the small time regime, \cite{LaurentL-21} gives the best possible control cost estimate for $S$ equal to a spherical cap.
Our Theorem~\ref{thm:observability-control-cost-estimate} extends this upper bound to all $S$ of positive measure.

In order to understand the small time asymptotics we consider a small spherical cap $S$.

\begin{example}\label{ex:control-costs-ball}
	Let $S=K(x_0,r)$ for some $r\in (0,1)$ and some $x_0\in\sphere_1^{d-1}$.
	Then $S$ is $\gamma$-thick with $\gamma=\abs{S}/\abs{\sphere^{d-1}}$.
	For simplicity, we assume that $|\log\gamma|\geq 1$.
	Then, for $T<1$, the constant $C_\obs$ in Theorem~\ref{thm:observability-control-cost-estimate} (with $R = 1$) satisfies
	\[
		C_{\obs}^2\leq\frac{c_2}{T}\exp\Bigl( \frac{2c_2\abs{\log\gamma}^2}{T}\Bigr).
	\]
	Comparing this inequality to the lower bound (valid for sufficiently small $r$, resp.~$\gamma$)
	\[
		C_\obs^2\geq \tilde{C}'\exp\Bigl( \tilde{C}\frac{\abs{\log\gamma}^2}{T}\Bigr),\quad  \tilde{C}, \tilde{C}'>0,
	\]
	given
	in \cite[Theorem~1.2]{LaurentL-21}, we see that our bound is best possible in this regime.
\end{example}

Let us emphasize that in the small time regime, upper and lower bounds (of different type)
for the constant $C_\obs$ were already derived in \cite{Miller-04}. There compact Riemannian manifolds are considered
and the special case of the sphere is discussed after Theorem 2.3.

While the above mentioned papers pursue various strategies of proof, there is a series of works that are methodological closely related to ours in that they
pursue Kovrijkine's original approach and implement it in a number of different contexts to obtain a spectral inequality of the type \eqref{eq:Logvinenko-Sereda-operator}, e.g.,
\begin{itemize}
	\item if $M$ is the torus or an infinite strip and $H=-\Delta$ is an appropriate self-adjoint realization of the Laplacian \cite{EgidiV-20, Egidi-21};
	\item if $M=\RR^d$ and $H=-\Delta+|x|^2$ is the harmonic oscillator \cite{BeauchardJPS-21,MartinPS-22,DickeSV-23a} (and \cite{DickeSV-23b} for a treatment
	of the partial harmonic oscillator);
	\item if the domain $M$ admits an appropriate covering and $H$ is an operator such that every function in the range of the spectral projections
	satisfies a Bernstein-type inequality \cite{EgidiS-21}.
\end{itemize}
See also \cite{GhobberJ-13} for an adaptation of Kovrijkine's strategy for the Fourier-Bessel transform.

%
%

\section{Proofs}\label{sec:Logvinenko-Sereda}

In contrast to previous proofs implementing Kovrijkine's approach, we are not dealing with an Euclidean setting here.
We therefore adapt the geometric constructions to the sphere, as already seen in Definition~\ref{def:thick}.
More precisely, we replace Euclidean balls (or rectangles) by spherical caps and Euclidean line segments by \emph{spherical line segments}.
A spherical line segment starting at a point $p\in \sphere^{d-1}_R$ in direction $v\in\sphere^{d-1}_R$ with $p\cdot v=0$
is a set $I\subset\sphere^{d-1}_R$ that is the image of the restriction of the curve
\be\label{eq:curve}
	\kappa=\kappa_v \colon [0,2\pi]\to\sphere^{d-1}_R,\quad t\mapsto\cos(t)p+\sin(t)v
\ee
to an interval $[0,l]$, $l>0$; in other words $I=\trace(\kappa\restriction_{[0,l]})$.
Given a spherical line segment and a measurable set $M\subset\sphere^{d-1}_R$, it is natural to define the \emph{arc length measure} of the set $I\cap M$ by
\be\label{eq:arc-length-measure}
	\arc{I\cap M} = \int_0^l \indic_{I\cap M}(\kappa(t))|\kappa'(t)|\Diff{t}
	= R\int_0^l \indic_{I\cap M}(\kappa(t))\Diff{t}.
\ee

We use spherical line segments to explicitly construct a variant of spherical polar coordinates centered at some point $p$ on the sphere.
To this end, note that there is a one-to-one correspondence $\Phi\colon \sphere^{d-2}_R \to \{v\in\sphere^{d-1}_R\colon p\cdot v = 0\}$
between the possible directions of spherical line segments starting at $p$ and the sphere $\sphere^{d-2}_R$.
Hence, if we let $\kappa_v$ be the curve $\kappa$ from \eqref{eq:curve} starting at $p$ in direction $\Phi(v)$, integrating over all those
possible directions we derive the \emph{spherical polar coordinates formula}
\be\label{eq:polar-coordinates-formula}
	\int_{\sphere^{d-1}_R}f\Diff{\sigma^{d-1}_R} = \frac{R}{2} \int_{\sphere^{d-2}_R}\int_0^{2\pi}f(\kappa_v(t))|\sin(t)|^{d-2}\Diff{t}\Diff{\sigma^{d-2}_R}(v),
\ee
where $\sigma^{d-k}_R$ denotes the surface measure of $\sphere^{d-k}_R$, $k < d$.

\subsection{Reduction to spherical line segments}

We now reduce our considerations to spherical line segments.
Our first lemma is based on the classic result \cite[Theorem~1.5]{Nazarov-94}, which is called there \emph{Turan Lemma} in reference to \cite{Turan-84}.
We recall it for convenience.

\begin{lemma}\label{lem:Turan}
	For all $n\in\NN$, all coefficients $\beta_1,\ldots, \beta_n\in\CC$, $\lambda_1,\ldots, \lambda_n\in\RR$,
	and all measurable sets $A\subset [0,1]$ with $|A|>0$ the function
	$r(x)=\sum_{k=1}^n \beta_k\euler^{i\lambda_k x}$ satisfies
	\be\label{eq:Turan}
		\norm{r}_{L^\infty([0,1])}\leq \Bigl(\frac{316}{|A|}\Bigr)^{n-1}\norm{r}_{L^\infty(A)}
		.
	\ee
\end{lemma}

The proof of our first lemma is based on the observation that restrictions of spherical polynomials to spherical line segments satisfy the
assumptions of Lemma~\ref{lem:Turan}.

\begin{lemma}\label{lem:Lp-norm-by-sup}
	Let $1\leq q\leq\infty$, let $K$ be a spherical cap of radius $a>0$ and let $f$ be a spherical polynomial of degree at most $\cN\in\NN$.
	Then there exists a point $p\in K$ such that 
	$|f(p)|\geq \norm{f}_{L^q(K)} \, |K|^{-1/q}$ 
	and for
	all spherical line segments $I\subset K$ starting at $p$ and all measurable subsets $M\subset\sphere^{d-1}_R$ such that $\arc{M\cap I}>0$ we have
	\be\label{eq:Lp-norm-by-sup}
		\norm{f}_{L^q(K)}\leq |K|^{1/q} \Bigl( \frac{316\cdot\arc{I}}{\arc{I\cap M}}\Bigr)^{2\cN}\sup_{M\cap K}|f|.
	\ee
\end{lemma}

\begin{proof}
	A simple proof by contradiction shows the existence of the point $p\in K$.
	
	Let $I\subset K$ be a spherical line segment starting at $p$, such that $\arc{M\cap I}>0$.
	If $v\in\sphere^{d-1}_R$ denotes the direction of $I$, then the curve
	$
		\kappa(t) = \cos(\arc{I}t/R)p+\sin(\arc{I}t/R)v
	$ with $t\in [0,1]$
	parameterizes $I$.
	Since $p=\kappa(0)$, we have
	\be\label{eq:line-segment-lower-bound}
		\norm{f\circ \kappa}_{L^\infty([0,1])} \geq |(f\circ\kappa)(0)| = |f(p)| \geq \norm{f}_{L^q(K)}/|K|^{1/q}.
	\ee
	Let $A=\{t\in [0,1]\colon \kappa(t)\in M\}$ be the set of all parameters for which $\kappa$ takes values in $M$.
	Clearly,
	\be\label{eq:line-segment-upper-bound}
		 \sup_{M\cap I}|f| = \norm{f\circ \kappa}_{L^\infty(A)}.
	\ee
	Furthermore, since the set $A$ satisfies
	\[
		|A| = \int_0^1 \indic_M(\kappa(t))\Diff{t} = \frac{1}{\arc{I}}\int_0^1 \indic_M(\kappa(t))|\kappa'(t)|\Diff{t}= \arc{I\cap M}/\arc{I}>0,
	\]
	the assumptions of Lemma~\ref{lem:Turan} are satisfied for $r=f\circ\kappa$ if there are
	finite sequences $(\beta_k)_k$ and $(\lambda_k)_k$ such that
	\be\label{eq:Turan-Darstellung}
		(f\circ\kappa)(t) =  \sum_{k=1}^{n} \beta_k\euler^{i\lambda_k t}
		\quad \text{for some} \quad n\in\NN
		.
	\ee
	
	The last identity is a consequence of the fact that $f$
	is the restriction to the sphere of a polynomial $P$ of degree at most $\cN$:
	There are coefficients $(b_\alpha)_{|\alpha|\leq \cN}$ such that $P(x) = \sum_{|\alpha|\leq \cN} b_\alpha x^\alpha$ satisfies
	$P = f$ on $\sphere^{d-1}_R$.
	Setting $v_j=v\cdot e_j$ and $p_j=p\cdot e_j$, we write
	\[
		\prod_{j=1}^d\bigl[ \kappa_j (t)\bigr]^{\alpha_j}
		= \prod_{j=1}^d\bigl[ p_j\cos(\arc{I}t/R)+v_j\sin(\arc{I}t/R)\bigr]^{\alpha_j}
	\]
	leading to
	\[
		(f\circ\kappa)(t) = \sum_{|\alpha|\leq \cN} c_\alpha \prod_{j=1}^d\bigl[ p_j\cos(\arc{I}t/R)+v_j\sin(\arc{I}t/R)\bigr]^{\alpha_j}.
	\]
	Finally, writing
	\[
		p_j\cos(\arc{I}t/R)+v_j\sin(\arc{I}t/R)
		=
		p_j \,\Re\bigl(\euler^{i\arc{I}t/R}\bigr) +v_j \,\Im\bigl(\euler^{i\arc{I}t/R}\bigr)
	\]
	and using the binomial theorem we find finite sequences such that equation \eqref{eq:Turan-Darstellung}
	holds with $n=2\cN+1$. (In fact, all $\lambda_k$ are integer multiples of $\arc{I}t/R$.)
	
	Applying Lemma~\ref{lem:Turan} now shows
	\[
		\norm{f\circ \kappa}_{L^\infty([0,1])} \leq \Bigl(\frac{316}{|A|}\Bigr)^{2\cN}\norm{f\circ \kappa}_{L^\infty(A)}
		 = \Bigl(\frac{316\arc{I}}{\arc{I\cap M}}\Bigr)^{2\cN}\norm{f\circ \kappa}_{L^\infty(A)}.
	\]
	Combining this inequality with \eqref{eq:line-segment-lower-bound} and \eqref{eq:line-segment-upper-bound} finishes the proof.
\end{proof}

The above lemma holds for all possible directions of the spherical line segments $I$.
Using the polar coordinates formula \eqref{eq:polar-coordinates-formula} we optimize the right hand side of \eqref{eq:Lp-norm-by-sup}, i.e.,
choose the direction of $I$ such that the quotient $\arc{I}/\arc{I\cap M}$ is small.

\begin{lemma} \label{lem:direction}
	Let $K$ be a spherical cap, let $p\in K$, and let $M\subset K$ be a measurable set satisfying $|M|>0$.
	Then there is a spherical line segment $I\subset K$ starting at $p$ such that
	\[
		\frac{\arc{I}}{\arc{I\cap M}} \leq c_4 \cdot \frac{|K|}{|K\cap M|}
		.
	\]
\end{lemma}

\begin{proof}
	Integrating with respect to the spherical polar coordinates centered at $p$, we have
	\[
		|M\cap K| = \frac{R}{2} \int_{\sphere^{d-2}_{R}}\int_{0}^{2\pi} \indic_{M\cap K}(\kappa_v(t))|\sin^{d-2}(t)|\Diff{t}\Diff{\sigma^{d-2}_R(v)}
	\]
	and there exists $v_0\in \sphere^{d-2}_{R}$ such that
	\[
		|M\cap K|\leq  \frac{R}{2}\cdot\sigma^{d-2}_R(\sphere^{d-2}_{R})\int_{0}^{2\pi} \indic_{M\cap K}(\kappa_{v_0}(t))|\sin^{d-2}(t)|\Diff{t}.
	\]
	Denoting the integral on the right hand side by $J$, it is clear that
	\[
		J = \int_{0}^{\pi} \indic_{M\cap K}(\kappa_{v_0}(t))|\sin^{d-2}(t)|\Diff{t}
		+ \int_{0}^{\pi} \indic_{M\cap K}(\kappa_{-v_0}(t))|\sin^{d-2}(t)|\Diff{t}
		=:J_1+J_2.
	\]
	This shows that there are two possible spherical line segments starting at $p$ and we choose the one that sees a larger part of the set $M$.
	More precisely, we either have $J\leq 2J_1$ or $J\leq 2J_2$; without loss of generality we  suppose $J\leq 2J_1$.
	Let $l = \sup\{t\in [0,\pi]\colon \kappa_{v_0}(t) \in K\}$ and
	$I = \trace\bigl(\kappa_{v_0}\restriction_{[0,l]}\bigr)$.
	We have $Rl=d_R(p,\kappa_{v_0}(l))\leq \diam(K)$ and thus $l\leq \diam(K)/R$.
	Hence, using this upper bound for $l$ and the simple bound $\sin(t)\leq t$, we obtain
	\eqs{
		|M\cap K|
		& \leq R\sigma^{d-2}_R(\sphere^{d-2}_R)J_1 \\
		& \leq R\sigma^{d-2}_R(\sphere^{d-2}_R)\int_{0}^l \indic_{M\cap K}(\kappa_{v_0}(t))t^{d-2}\Diff{t}\\
		& \leq \sigma^{d-2}_R(\sphere^{d-2}_R) (\diam(K)/R)^{d-2} \arc{I\cap M}
		.
	}
	Since $\arc{I}\leq \diam(K)$ and $\diam(K)^{d-1} = c_3 \abs{K}$,
	we have thus shown
	\be\label{eq:ratio-inequality}
		\frac{\arc{I\cap M}}{\arc{I}}\geq \frac{\abs{M\cap K}}{c_4 \abs{K}}
		.
		\qedhere
	\ee
\end{proof}

\subsection{Covering argument and the proof of Theorem~\ref{thm:logvinenko-sereda-sphere}}\label{ssec:proof-LS}

Let $f$ be a spherical polynomial of degree at most $\cN$.
Combining Lemma~\ref{lem:Lp-norm-by-sup} and~\ref{lem:direction}, we have already shown
\be\label{eq:sup-Lq}
	\norm{f}_{L^q(K)}\leq |K|^{1/q} \Bigl( c\cdot \frac{|K|}{|K\cap M|}\Bigr)^{2\cN}\sup_{M\cap K}|f|
\ee
with $c = 316c_4$, where $M$ is any measurable set satisfying $|K\cap M| > 0$.
Here we choose $M$ as the set of points inside the spherical cap $K$ where $|f|$ is small relative to its
own $L^q$-norm on $K$ and the measure of $S$.
That is,
\[
	M = M_{f,S}=\Bigl\{ x\in K\colon |f(x)| < |K|^{-1/q}\Bigl(\frac{|K\cap S|}{2c|K|} \Bigr)^{2\cN}\norm{f}_{L^q(K)} \Bigr\}.
\]
In what follows, we assume without loss of generality that $M_{f,S}\neq\emptyset$.
Since $M_{f,S}$ is an open set, we then have $|M_{f,S}\cap K|=|M_{f,S}|>0$.
In particular, we are in the position to apply inequality~\eqref{eq:sup-Lq} and the definition of $M_{f,S}$ to obtain
\[
	\norm{f}_{L^q(K)}\leq |K|^{1/q} \Bigl( c\cdot \frac{|K|}{|K\cap M_{f,S}|}\Bigr)^{2\cN} \sup_{M_{f,S}} |f|
	\leq \Bigl( \frac{|K\cap S|}{2|M_{f,S}|} \Bigr)^{2\cN} \norm{f}_{L^q(K)}
	.
\]
Since $M_{f,S}\neq\emptyset$, we have $\norm{f}_{L^q(K)}>0$ and therefore $|K\cap S| \geq 2|M_{f,S}|$.
Hence, $|(K\cap S)\setminus M_{f,S}| \geq |K\cap S|/2$ by the last inequality.
Since $f$ is small on $M_{f,S}$, we estimate $\norm{f}_{L^q(K\cap S)} \geq \norm{f}_{L^q((K\cap S)\setminus M_{f,S})}$
and use the lower bound for $|f|$ on $K\setminus M_{f,S}$.
Thereby we obtain
\bes
	\norm{f}_{L^q(K\cap S)} \geq\Bigl( \frac{|K\cap  S|}{2c|K|} \Bigr)^{2\cN+1/q}\norm{f}_{L^q(K)}.
\ees

This is a local variant of the Logvinenko-Sereda-Kovrijkine type inequality.
Indeed, recall that there is a radius $a>0$ such that $|K\cap S|\geq\gamma|K|$ for all spherical caps $K$ of radius $a$.
Hence, for all such $K$ the above inequality implies
\be\label{eq:capwise-Logvinenko-Sereda}
	\norm{f}_{L^q(K \cap  S)} \geq \Bigl(  \frac{\gamma}{2c} \Bigr)^{2\cN+1/q}\norm{f}_{L^q(K)}.
\ee

In order to complete the proof of Theorem~\ref{thm:logvinenko-sereda-sphere}, we choose a finite sequence of points
$(x_j)_{j=1}^m\subset\sphere^{d-1}_R$ such that the spherical caps $K_j=K(x_j,a)$ cover $\sphere^{d-1}_R$.
Moreover, we let
\be\label{eq:multiplicity}
	\kappa := \max_{y\in\sphere^{d-1}_R} \#\{j\in\{1,\dots,m\}\colon y \in K_j\}
\ee
be the maximal multiplicity of the covering, so that every point on the sphere is contained in at most $\kappa$-many caps $K_j$'s.
Using the local inequality \eqref{eq:capwise-Logvinenko-Sereda} on every cap $K_j$, we get
\[
	\norm{f}_{L^q(S)}^q \geq \frac{1}{\kappa}\sum_{j=1}^m\norm{f}_{L^q(S\cap K_j)}^q
	\geq \frac{1}{\kappa} \Bigl( \frac{\gamma}{2c}\Bigr)^{2\cN q+1} \norm{f}_{L^q(\sphere^{d-1}_R)}^q.
\]

We conclude the proof using the following lemma which is a simple consequence of results from \cite{FejesToth-64} and \cite[Theorem~1.1]{BoeroeczkyW-03}.

\begin{lemma}
	For any $d\in\NN\setminus\{1\}$ and all $R,a > 0$ there is a finite sequence of points $(x_j)_j \subset\sphere^{d-1}_R$
	such that $\sphere^{d-1}_R\subset\bigcup_j K(x_j,a)$ and the maximal multiplicity satisfies $\kappa \leq 400d\log d$.
\end{lemma}

With this lemma, we finally derive
\[
	\norm{f}_{L^q(S)} \geq \Bigl(\frac{\gamma}{c_5}\Bigr)^{2\cN +1/q}\norm{f}_{L^q(\sphere^{d-1}_R)}.
\]
This finishes the proof of Theorem~\ref{thm:logvinenko-sereda-sphere}.

\subsection{Proof of the spectral inequality in Corollary~\ref{cor:spectral-inequality-Laplace-Beltrami}}\label{ssec:spectral-inequality}

Recall that the eigenfunctions of the Laplace-Beltrami operator on the unit sphere are given by the spherical harmonics $Y_{\ell,k}$, $\ell\in\NN$, $k=1,\dots,n_\ell$,
where $\ell$ denotes the degree and $n_\ell$ the multiplicity.
It is well-known that $(Y_{\ell,k})_{\ell,k}$ forms an orthonormal basis of $L^2(\sphere^{d-1})$ and that
$-\Delta Y_{\ell,k} = \ell(\ell+d-2)Y_{\ell,k}$ for all $\ell\in\NN$ and $k=1,\dots,n_\ell$.
We let $Y_{\ell,k,R}(x)=Y_{\ell,k}(x/R)$ for $x\in\sphere^{d-1}_R$.
Then a simple scaling argument shows that the functions $(Y_{\ell,k,R})_{\ell,k}$ are the eigenfunctions of the Laplace-Beltrami on
the sphere $\sphere^{d-1}_R$.
More precisely,
\be\label{eq:spherical-harmonics-R-sphere}
	-\Delta Y_{\ell,k,R} =R^{-2}\ell(\ell+d-2)Y_{\ell,k,R}\quad\text{for all}\quad\ell\in\NN,\quad k=1,\dots,n_\ell
	,
\ee
and the family $\{Y_{\ell,k,R}\colon\ell\in\NN, k=1,\dots n_\ell\}$ forms an orthonormal basis of $L^2(\sphere^{d-1}_R)$.

In order to prove Corollary~\ref{cor:spectral-inequality-Laplace-Beltrami}, note that
each $(Y_{\ell,k,R})$ is a spherical polynomial of degree at most $\ell$, so that \eqref{eq:spherical-harmonics-R-sphere}
shows that for all $Y_{\ell,k,R}\in \Ran P_{-\Delta}((-\infty,E])$ we have $\ell\leq RE^{1/2}$.
Hence, $f\in \Ran P_{-\Delta}((-\infty,E])$ (being a finite linear combination of $Y_{\ell,k,R}$ with $\ell\leq RE^{1/2}$)
is a spherical polynomial of degree at most $RE^{1/2}$.

Thus, Theorem~\ref{thm:logvinenko-sereda-sphere} implies \eqref{eq:Logvinenko-Sereda-Laplace-Beltrami}
and Corollary~\ref{cor:spectral-inequality-Laplace-Beltrami}.

%
%

\section{Discussion of the covering argument}

For every $\gamma_0 > 0$ and every $\gamma_0$-thick set $S\subset \sphere^{d-1}_R$ there exists a scale $a_0>0$ such that $|S \cap K| \geq \gamma_0 |K|$
for all spherical caps of radius $a_0$.
On the other hand, we have $|S| > 0$ and hence $\gamma_1 := |S|/|\sphere^{d-1}_R| > 0$.
This results in two competing estimates in Theorem~\ref{thm:logvinenko-sereda-sphere} with constants
\be\label{eq:comparison}
	\Bigl(\frac{c_1}{\gamma_0}\Bigr)^{2\cN+1/q}
	\quad \text{and} \quad
	\Bigl(\frac{c_1}{\gamma_1}\Bigr)^{2\cN+1/q},
\ee
respectively.
Furthermore, for a covering $(K_j)_j$ of $\sphere^{d-1}_R$ with spherical caps of radius $a_0$ we have
\bes
	|\sphere^{d-1}_R| \leq \sum_j |K_j| \leq \sum_j \frac{1}{\gamma_0}|S\cap K_j|
	\leq \frac{\kappa}{\gamma_0} |S\cap\cup_j K_j| = \frac{\kappa}{\gamma_0} |S|
	,
\ees
where $\kappa$ is the multiplicity from \eqref{eq:multiplicity} (depending only on the dimension).
Hence, $1/\gamma_1 \leq \kappa/\gamma_0$ and the second constant in \eqref{eq:comparison}
can be brought into the form of the first one by increasing $c_1$, that is
\be
\Bigl(\frac{c_1}{\gamma_1}\Bigr)^{2\cN+1/q}
	\leq
	\Bigl(\frac{\kappa c_1}{\gamma_0}\Bigr)^{2\cN+1/q}
\ee
and $c_1$ is indeed allowed to depend on the dimension (only).
This illustrates that the covering argument above merely improves the constant $c_1$, but not the qualitative features of the control cost bound, when compared to the situation if we consider exclusively $K = \sphere^{d-1}_R$.

This stands in contrast to the Euclidean estimate (cf.~Theorem~\ref{thm:Logvinenko-Sereda-RRd}), where thick sets with small scales $a$ yield a
substantially better bound since $a$ also appears as a (small) factor in the exponent.
We are not (yet) able to reproduce this exponential scaling with respect to $a$ in the spherical setting (cf.~Theorem~\ref{thm:logvinenko-sereda-sphere}).
For this purpose it would be necessary to replace inequality~\eqref{eq:Lp-norm-by-sup} by a bound where the exponent exhibits a factor $a$.

\subsection*{Acknowledgments.}
The authors would like to thank M.~K\"amper for a careful reading of the manuscript and A.~Seelmann for many fruitful discussions.
A.D.~has been partially supported by the DFG grant VE 253/10-1 entitled
\emph{Quantitative unique continuation properties of elliptic PDEs with variable 2nd order coefficients and applications in control theory, Anderson localization, and photonics}.
%
%

\newcommand{\etalchar}[1]{$^{#1}$}
\def\polhk#1{\setbox0=\hbox{#1}{\ooalign{\hidewidth
  \lower1.5ex\hbox{`}\hidewidth\crcr\unhbox0}}}

\end{document}